\documentclass[11pt, reqno]{amsart}
\usepackage{amsmath}
\usepackage{amsthm}
\usepackage{amssymb}
\usepackage{amscd}
\usepackage{amsbsy}
\usepackage{epsfig}
\usepackage{graphicx}
\usepackage{psfrag}
\usepackage{bbm}
\usepackage{hyperref}
\usepackage[UKenglish]{babel}% http://ctan.org/pkg/babel
\usepackage[UKenglish]{isodate}% http://ctan.org/pkg/isodate
\usepackage{cleveref}
\usepackage[pagewise]{lineno}

\theoremstyle{plain}% default
\newtheorem{theorem}{Theorem}[section]

\newtheorem{remark}[theorem]{Remark}
\newtheorem{proposition}[theorem]{Proposition}
\newtheorem{lemma}[theorem]{Lemma}

\theoremstyle{definition}
\newtheorem{definition}[theorem]{Definition}
\newtheorem{claim}{Claim}
\newtheorem*{claim*}{Claim}

\newtheorem*{lemma*}{Lemma}
\theoremstyle{definition}

\def\R{\ensuremath{\mathbb R}}
\def\N{\ensuremath{\mathbb N}}

\def\B{\ensuremath{\mathcal B}}
\def\M{\ensuremath{\mathcal M}}
\def\C{\ensuremath{\mathcal C}}
\def\L{\ensuremath{\mathcal L}}

\def\P{\ensuremath{\mathcal P}}

\def\F{\ensuremath{\mathcal F}}

\def\A{\ensuremath{\mathscr A}}

\def\ie{{\em i.e.}, }

\def\A{\mathcal{A}}

\def\eps{\varepsilon}

\numberwithin{equation}{section}

\begin{document}
\date{\today}

\author[M.~Todd]{Mike Todd}
\address{Mathematical Institute\\
University of St Andrews\\
North Haugh\\
St Andrews\\
KY16 9SS\\
Scotland} 
\email{m.todd@st-andrews.ac.uk}
\urladdr{https://mtoddm.github.io}    

\author[B.~Zhao]{Boyuan Zhao}
\address{Mathematical Institute\\
University of St Andrews\\
North Haugh\\
St Andrews\\
KY16 9SS\\
Scotland} 
\email{bz29@st-andrews.ac.uk}
\urladdr{https://boyuanzhao.github.io/}

\title{Countable Markov shifts with exponential mixing}
\date{\today}

\keywords{Countable Markov shifts, Exponential mixing}

\subjclass[2020]{37A25, 37D05, 37B10, 37D25}
\maketitle

\begin{abstract}
We give a set of equivalent conditions for a potential on a Countable Markov Shift to have strong positive recurrence, which is also equivalent to having exponential decay of correlations.  A key ingredient of our proofs is quantifying how the shift behaves at its boundary.
\end{abstract}

\section{Introduction}

A standard approach in ergodic theory and dynamical systems is to show that a dynamical system $f:X\to X$ with invariant probability measure $\mu$, asymptotically behaves like an independent, identically distributed process: how quickly this happens can be understood as the rate of decay of correlations, or rate of mixing, for $(X, f, \mu)$.  Generally, uniformly expanding systems with some compactness properties can be shown to have exponential rate decay of correlations.  Recent research has focussed on estimating the rate of decay of correlations for systems with non-uniformly expansion and/or lack of compactness.

Here we will consider dynamical systems $f:X\to X$ with measure $\mu=\mu_\phi$, which is an equilibrium state for some potential $\phi:X\to \overline{\R}$ .  Note that the first task in this setting is to show that such $\mu_\phi$ exists.  In the case of interval maps $f:I\to I$ with finitely many branches, on each of which $f$ is smooth, we might consider $\phi=-\log|Df|$, so an equilibrium state $\mu_\phi$ is an absolutely continuous (with respect to Lebesgue) invariant measure.  For example, if $f$ is unimodal and satisfies the Collet-Eckmann condition (exponential growth of derivative along the critical orbit), then $\mu_\phi$ exists and the system has exponential decay of correlations: indeed in \cite{NowSan98} these two conditions were shown to be equivalent, alongside a number of other conditions, one of which was that $\chi_{per}:=\inf\left\{\log|Df^n(x)|/n: f^n(x)=x, n\in \N\right\}$ has $\chi_{per}>0$. See \cite{Riv20} for related results in the multimodal case.

In the same dynamical setting, but in the case of H\"older potentials $\phi$, various conditions have been applied to ensure the existence of an equilibrium states, which also usually imply that the system also has exponential decay of correlations.  A classic such example is the condition that $\sup\phi-\inf\phi<h_{top}$, where $h_{top}$ denotes topological entropy, which was used for interval maps in \cite{HofKel82, BruTod08} for example (in the former a specification property was assumed, the latter dealt with multimodal maps of the interval).  A more general condition, which was used to show uniqueness of equilibrium states in \cite{DenKelUrb90},  is  $\sup\phi<P(\phi)$ where $P(\phi)$ is the pressure of $\phi$.   A weaker condition of \emph{hyperbolicity}, that there is some $n\in\N$ such that $\sup_{x\in X} S_n\phi(x)/n<P(\phi)$, was used in \cite{InoRiv12} in the case of complex rational maps, and in \cite{LiRiv14a} in the interval map case, to show existence and uniqueness of equilibrium states.  Then it was shown in \cite{LiRiv14b} that for many multimodal maps (satisfying a growth condition on critical orbits) \emph{all} H\"older potentials are hyperbolic. 
Related conditions for interval maps were given in \cite[Def 3.4]{LivSauVai98}, under which the authors  called the potential `contracting': added to a covering condition this implies exponential decay of correlations.

In the case of symbolic dynamical systems, nonuniform hyperbolicity can be a consequence of an infinite alphabet: here we will discuss Countable Markov Shifts.  In this setting, though we have a Markov structure which makes many issues straightforward, these structures are flexible, so that for example a very strong condition like $\sup\phi<P(\phi)$ is not sufficient to guarantee existence of equilibrium states (see Section~\ref{ssec:eg_not_spr}).  There are however a number of equivalent conditions guaranteeing exponential decay of correlations given in \cite{CyrSar09}, in particular Strong Positive Recurrence (SPR) and the Spectral Gap Property (SGP): these can be shown to imply exponential decay of correlations.  In this paper we combine the types of ideas used in that paper with those in  \cite{NowSan98} to show various equivalent conditions for (SPR), in particular that  $\chi_{per}(\phi):= \sup\left\{\frac1nS_n\phi(x): \sigma^n(x) = x, n\in \N\right\}$ has $\chi_{per}(\phi)<P(\phi)$.  The main idea is to use information on entropy at infinity, see \cite{IomTodVel22} (see also an equivalent notion in \cite{Buz10}) which, when it is not too large, can be used to control returns to compact parts of our system.  This idea can also be seen in \cite[Lemma 2.17, Theorem 7.14]{DobTod23}, where the entropy at infinity for the symbolic version of the finitely branched interval maps is zero.

We also give other conditions equivalent to the one mentioned above, one being `contraction at infinity': this rules out standard examples in the renewal shift, for example, where (SPR) fails due to very weak contraction in the boundary of the system.
We also give examples to show that that our conditions are sharp and moreover, other conditions related to those in \cite{NowSan98} are not equivalent to ours, see the discussion in Section~\ref{ssec:CEetc} and the examples in Section~\ref{sec:examples}.
We mention recent progress in the context of geodesic flows such as  \cite{IomRiqVel18, RiqVel19, SchTap21} which considered entropy and infinity and \cite{GouSchTap23} which linked this to an SPR condition.

The applications here are primarily to countable Markov shifts with finite topological entropy.  These include cases such as renewal shifts, $S$-gap shifts, $\beta$-shifts and symbolic models (such as Young towers or Hofbauer extensions) for interval maps with finitely many monotone branches, all of which have entropy at infinity equal to zero.  We also obtain results in the case when the topological entropy is finite, but the entropy at infinity is non-zero, see examples such as birth and death chains and others in \cite[Section 6]{IomTodVel22}, as well as the rich family of examples of bouquet shifts given in Section~\ref{sec:examples}.
Note that the main standard examples in the infinite topological entropy case are the full shift on $\N$ or a system with the BIP property: in these cases the condition $P(\phi)<\infty$ automatically implies SPR if $\phi$ is sufficiently regular.

The paper is organised as follows.  In Section~\ref{sec:defsthms} we give our main definitions and results, Theorems~\ref{thm:contract} and \ref{thm:UCS-SPR}.  In Section~\ref{sec:thmcontr} we prove Theorem~\ref{thm:contract} and in Section~\ref{sec:UCS-SPR} we prove Theorem~\ref{thm:UCS-SPR}.  Finally in Section~\ref{sec:examples} we give examples to show applications and sharpness of our results.

\textit{Acknowledgements.}  MT was partially supported by the FCT (Fundação para a Ciência e a Tecnologia) 
project 2022.07167.PTDC.  BZ acknowledges the support of a Chinese Scholarship Council grant.

\section{Definitions and main theorems}
\label{sec:defsthms}

In this section we give the general setting for our work and then state the main results.

\textit{Notation.}  If $A_n, B_n\in (0, \infty)$, we write $A_n\asymp B_n$ if there is some $C>0$ such that $B_n/C\le A_n\le CB_n$.  For a set $E$ we write $\#E$ to denote its cardinality.

\subsection{Countable Markov Shifts, basic setup}

Here we outline the Countable Markov Shift setting, for which we follow the works of Sarig, eg \cite{Sar99}.
Let $\mathcal{A}$ be a countable alphabet and $M$ an $\mathcal{A}\times\mathcal{A}$ transition matrix with entries 0 and 1.  Our shift space is given by $$\Sigma:=\left\{(x_0,x_1,\dots)\in\mathcal{A}^{\N_0}: M_{x_i,x_{i+1}}=1\text{ for }i\ge0\right\}.$$
Our dynamics is the usual shift $\sigma(x_0, x_1, \ldots) = (x_1, x_2, \ldots)$.

Let 
$$\Sigma_k:=\left\{w = w_0, \ldots, w_{k-1}: M_{w_i,w_{i+1}}=1 \text{ for }0\le i<k-1\right\}$$ be the set of all \emph{admissible words} with length $k$.  Let $\Sigma^*:= \cup_{k\ge 0} \Sigma_k$ (this means that we include the empty word here). Given $w= w_0, \ldots, w_{k-1}$, write $[w]_i=w_i$ for $0\le i\le k-1$. For each $v\in \Sigma_k, w\in \Sigma_{k'}$, if $M_{v_{k-1},w_0}=1$ the concatenation $vw$ is also admissible in $\Sigma_{k+k'}$; any point in $\Sigma$ can be viewed as an infinite concatenation of finite words. 

 We use $|\cdot|$ to denote the length of an admissible word: so if $w\in \Sigma_k$ then $|w|=k$. 
Given $x_0,\dots,x_{k-1}\in \Sigma_k$, the corresponding $k$-\emph{cylinder} is 
$$[x_0,\dots,x_{k-1}]:=\left\{y\in\Sigma:y_0=x_0, \ldots, y_{k-1}= x_{k-1}\right\}.$$ 
The set of all such $k$-cylinders is denoted by $\mathcal{C}_k$.  
We also define the symbolic metric $d(\cdot,\cdot)$ by
$$d(x,y)=2 ^{-t(x,y)},\hspace{3mm} t(x,y):=\min\{n\ge0:x_n\neq y_n\},$$
so $k$-cylinders have diameter $2^{-k}$.

The system is said to be \emph{topologically  transitive }if for each $a,b\in\mathcal{A}$, there exists $N_{ab}$ such that $[a]\cap\sigma^{-N_{ab}}[b]\neq\emptyset;$ the system is \emph{topologically mixing} if, moreover, 
$[a]\cap\sigma^{-n}[b]\neq\emptyset$ for any $n\ge N_{ab}$.

If a point $x$ is periodic with period $k$, it can be denoted by $x=(\overline{x_0,\dots,x_{k-1}}).$

Define the  \emph{$k$-th variation} of a potential $\phi:\Sigma\rightarrow\R$ by $$var_k(\phi):=\sup\left\{|\phi(x)-\phi(y)|:x_0=y_0,\dots,x_{k-1}=y_{k-1}\right\}.$$ 
Then $\phi$ is said to have \emph{summable variations} if the \emph{distortion constant} $B_\phi:=\sum_{k\geq2}var_k(\phi)$ is finite, and \emph{weakly H\"older continuous} if there exists $\theta\in(0,1)$ and $C>0$ such that $var_k(\phi)\le C\theta^k$ for all $k\ge 2$.

\subsection{Pressure, recurrence and contraction}

 For each $a\in\mathcal{A}$, we let $\varphi_a(x):=\mathbbm{1}_{[a]}\inf\{k\geq1:x_k=a\}$ be the \textit{first return time} to $[a]$ and
define $$Z_n(\phi,a):=\sum_{\sigma^nx=x}e^{S_n\phi(x)}\mathbbm{1}_{[a]}(x),\hspace{3mm}Z_n^*(\phi,a):=\sum_{\sigma^nx=x}e^{S_n\phi(x)}\mathbbm{1}_{[\varphi_a=n]}(x).$$
  The \emph{Gurevich pressure} of $\phi$ is then given by $$P_G(\phi):=\limsup_{n\rightarrow\infty}\frac{1}{n}\log{Z_n(\phi,a)}$$
    which is easily seen to be independent of the state $a$ under topological transitivity, and if the system is topological mixing, the $\limsup$ can be replaced by $\lim$. The \emph{topological entropy} $h_{top}$ of $(\Sigma,\phi)$ is the quantity $P_G(0)$. By the Variational Principle, see \cite[Theorem 2.10]{IomJorTod15} for a version in the generality considered here,
    $$P_G(\phi)=P(\phi):=\sup\left\{h_\nu(\sigma)+\int\phi\,d\nu:\nu\in\mathcal{M}(\Sigma),\int\phi\,d\nu>-\infty\right\},$$
    where $\mathcal{M}(\Sigma)$ is the space of all $\sigma$-invariant Borel probability measures on $\Sigma$. A probability measure is said to be an \emph{equilibrium state} for $\phi$ if it realises the supremum above. If $m$ is a measure such that whenever $A\subset\Sigma$ is measurable and $\sigma:A\to \sigma(A)$ is bijective, $m(\sigma(A))= \int_Ae^{-\phi}~dm$, then $m$ is \emph{$\phi$-conformal}.

    The potential
    $\phi$ is said to be \emph{recurrent} if $\sum_{n\ge1}e^{-nP(\phi)}Z_n(\phi,a)$ diverges and \emph{transient} otherwise. In the recurrent case, $\phi$ is further said to be \emph{positive recurrent }if $\sum_{n\ge1}ne^{-nP(\phi)}Z_n^*(\phi,a)<\infty$, and \emph{null recurrent} otherwise.  Note that there are $(\phi-P(\phi))$-conformal measures when $\phi$ is recurrent, as in \cite[Theorem 1]{Sar01}.

We finish this subsection by highlighting the key definitions of this paper.  The following notion uses inducing: given $a\in \A$, consider $\overline\sigma_a(x) = \sigma^{\varphi_a(x)}(x)$ for $x\in [a]$.  Then the induced version of $\phi$ is $\overline\phi =\overline\phi_a:= \sum_{i=0}^{\varphi_a-1}\phi\circ\sigma^i$.  The pressure of $\overline\phi$ for the induced system $([a], \overline\sigma_a)$ is given as above.

\begin{definition}[Strong Positive Recurrence]\label{def:SPR}
    Following the conventions and notation in \cite{Sar01}, for a given state $a\in\mathcal{A}$ define the quantity 
    $$p_a^*[\phi]:=\sup\left\{p\in\R:P(\overline{\phi+p})<\infty\right\},$$
    $$\Delta_a[\phi]:=P\left(\overline{\phi+p_a^*[\phi]}\right).$$

    The shift system $(\Sigma,\sigma,\phi)$ is said to be \emph{strongly positive recurrent (SPR)} if $\Delta_a[\phi]>0$ for some $a\in\mathcal{A}$. 
        \end{definition}

    Equivalently, for some $a\in\mathcal{A}$,
    \begin{equation}\label{eq:SPRequiv}
        \limsup_{n\rightarrow\infty}\frac{1}{n}\log Z_n^*(\phi,a)<P(\phi),
    \end{equation}
    see for example \cite[Lemma 2.1]{RuhSar22} (note that SPR potentials are also positive recurrent).

Next, we assume our system satisfies a condition that regulates the excursions to infinity.  We first require some more notation.  Here we rewrite $\A= \N$ (while the setup for the following notions do rely on the precise identification of $\A$ with $\N$ here, the limiting quantities obtained from them do not).  Then for each $q\in \N$, define $$[\le q]:=\bigcup_{a=1}^{q}[a],\text{ and }[>q]:=\bigcup_{a>q}[a].$$

\begin{definition}[Entropy and contraction at infinity] \label{def: no escape of mass}
    For each $n$, $M$ and $q$, define the set of $n+1$-cylinders 
    $$B(n,M,q):=\left\{[x_0,\dots,x_n]\in\C_{n+1}:x_0,x_{n}\leq q,\, \#\{x_k\leq q\}\leq \frac{n+1}{M}\right\}$$
    and write $z_n(M, q):= \#B(n,M,q)$.
   The \emph{entropy at infinity} $h_\infty$ as in \cite[Def 1.2]{IomTodVel22} is defined via
    $$h_{\infty}(M,q):=\limsup_{n\rightarrow\infty}\frac{1}{n}\log z_n(M, q),\hspace{3mm}h_{\infty}(q):=\liminf_{M\rightarrow\infty}h_{\infty}(M,q),$$
    $$h_{\infty}:=\liminf_{q\rightarrow\infty}h_{\infty}(q).$$
    Also, define the following quantities: 
    $$z_{\phi, n}(M, q) := \sup \left\{\frac1nS_n\phi(x): x\in B,\text{ for some }B\in B(n,M,q)\right\},$$
$$\delta_{\phi, \infty}(M, q) := \limsup_{n\to\infty}z_{\phi, n}(M, q), \ \delta_{\phi, \infty}(q):= \liminf_{M\to\infty}\delta_{\phi, \infty}(M,q),$$
$$\delta_{\phi, \infty}: = \liminf_{q\to\infty}\delta_{\phi, \infty}(q).$$
Then the system is said to have \emph{contraction at infinity (CI)}, if $\delta_{\phi,\infty}<P(\phi)$.

Notice that, by definition, $h_\infty(M,q)$, $\delta_{\infty,\phi}(M,q)$ are both monotone decreasing in $M$ and $q$, therefore the $\liminf$'s in the definition of $h_\infty$ and $\delta_{\phi,\infty}$ are in fact limits (which includes $-\infty$). 
\end{definition}

Finally we give a topological property of our shift, an idea first introduced in \cite[Definition 4.9]{IomVel21}.

\begin{definition}[The $\F$ property]\label{def: F-property}
 Our system has the \emph{$\F$-property} if for any state $a\in\mathcal{A}$ and every $n\in\N$, the number of periodic points in $[a]$ with period $n$ is finite. \end{definition}

Examples of when the $\F$-property holds include when $(\Sigma, \sigma)$:
\begin{itemize}
\item is locally compact (i.e. for every $i\in\mathcal{A}$, $\sum_{j\in\mathcal{A}}M_{i,j}<\infty$);
\item has $h_{top}<\infty$.
\end{itemize}

From the examples given in Section~\ref{sec:examples}, we can see that a system may satisfy the $\F$-property whilst failing all the above conditions.  Note that if $\phi$ is uniformly bounded from below and $P(\phi)<\infty$ then the $\F$-property must also hold.

\begin{remark}
The main fact we will use about the $\F$-property in this paper is that, by topological transitivity, given $q\in \N$ and $N\in \N$, there are finitely many paths starting and ending in $[\le q]$ of length $N$.
\end{remark}

\begin{remark}
If the $\F$-property did not hold then the notion of entropy at infinity here does not make sense in the following way.  In this case there will be some $q\in \N$ with infinitely many loops of some length $N$ based at $q$.  Thus there will be infinitely many 1-cylinders of arbitrarily high index $q'\in \N$ which, when $M>N$ will not feature in $B(n, M, q)$: hence $h_\infty$ does not see these cylinders at all, indeed $z_n(M, q)$ could be zero for all large $n, M, q$.  This issue would also appear in the definition of $h_\infty$ in \cite{Buz10}.

We also note that, on the other hand, when the $\F$-property does hold then for any $q, M$, if $n$ is large enough then $B(n, M, q)\neq \emptyset$.  
\label{rem:Fprop}
\end{remark}

\subsection{Main Theorems}

\begin{theorem}\label{thm:contract}
    Let $(\Sigma,\phi)$ be a topologically transitive CMS with the $\F$-property, $\phi$ a potential with summable variations and the pressure $P(\phi)<\infty$. Then the following are equivalent.
    \begin{enumerate}
    \item[(UCS)]  \label{UCS}Uniformly contracting structure:
    $$\chi_{per}(\phi):= \sup\left\{\frac1nS_n\phi(x): \sigma^n(x) = x\right\}<P(\phi);$$
    \item[(CRC)] \label{CRC} Compact returns contract: Given $q\in \N$, there exist $C_q\in\mathbbm{R}$ and $\lambda_q>P(\phi)$ such that: for all $n\in\N$, if $x\in \Sigma$ has $x_0, x_n\le q$ then $$S_n\phi(x)\le  C_q - n\lambda_q;$$
    \item[(CI)] \label{CI}Contraction at infinity: $\delta_{\phi, \infty}<P(\phi)$.
    \end{enumerate}
\end{theorem}

All of the conditions here are new in this context as far as we are aware.  As mentioned in the introduction, (\hyperref[UCS]{UCS}) can be compared to $\chi_{per}>0$ in \cite{NowSan98} since in that case $\phi=-\log|Df|$, so $\chi_{per}=-\chi_{per}(\phi)$,  and moreover $P(\phi)=0$.  We also note that it is easy to show that $\chi_{per}(\phi)=\sup\{\int\phi~d\mu:\mu\in \M(\Sigma)\}$.  We use the word `contraction' for our potential in line with \cite{LivSauVai98}.
If the assumptions of \Cref{thm:contract} are slightly strengthened, then the following holds.

\begin{theorem}\label{thm:UCS-SPR}
Let $(\Sigma,\sigma,\phi)$ be a topologically transitive CMS, $\phi$ a potential with summable variations such that $P(\phi)<\infty$, and entropy at infinity $h_{\infty}=0$. 
 Then (\hyperref[UCS]{UCS}) holds if and only if \hyperref[def:SPR]{SPR} holds.
\end{theorem}

Note that we are also able to prove related results when $h_\infty>0$, such as Theorem~\ref{thm:CI-SPR}

Theorem~\ref{thm:UCS-SPR} provides a new characterisation of SPR which, in the weakly H\"older potential case, is known to be equivalent to various other conditions, for example the spectral gap property, which gives exponential decay of correlations/mixing (see \cite{CyrSar09}).  We also note that in \cite[Theorem 2.2']{CyrSar09} \hyperref[def:SPR]{SPR} was shown to be open and dense in appropriate topologies.  So in the setting of the above theorem  (\hyperref[UCS]{UCS}) is also open and dense.  We observe that openness of such a condition is elementary, but denseness requires more work.
 
 \subsection{Other conditions on contraction}
 \label{ssec:CEetc}
As in the introduction, inspired by, for example \cite{NowSan98}, one might propose further conditions and hope that they are equivalent to those in our main theorems.   Here we suppose $P(\phi)=0$.

The following conditions involve hitting compact parts

\begin{itemize}

\item[(A)]\label{A} There exist $q\in \N$, $C\in\mathbbm{R}$, $\eps>0$ such that if $x\in \Sigma$ has $x_{n}\le q$ then $S_n\phi(x)\le  C - n\eps$.

\item[(B)]\label{B} There exist $q\in \N$, $C\in\mathbbm{R}$, $\eps>0$ such that if $x\in \Sigma$ has $x_0\le q$, then $S_n\phi(x)\le  C - n\eps$.

\end{itemize}

Here (\hyperref[A]{A}) can be interpreted as `orbits which land in a compact part experience contraction' and (\hyperref[B]{B}) can be interpreted as `orbits which start in a compact part experience contraction'.

The following condition involves avoiding compact parts

\begin{itemize}

\item[(C)]\label{C} There exist $q\in \N$, $C\in\mathbbm{R}$, $\eps>0$ such that if $x\in \Sigma$ has $x_k> q$ for $k=0, \ldots, n-1$, then $S_n\phi(x)\le  C - n\eps$.

\end{itemize}

Supposing there is a $\phi$-conformal measure $m_\phi$, we might also hope to show that measures of cylinders decay exponentially: below we give a weak condition of this type, which by conformality of $m_\phi$, a measure we assume to exist here, is essentially equivalent to (\hyperref[A]{A}) above.

\begin{itemize}

\item[(D)]\label{D} There exist $q\in \N$,  $\lambda>1$ and $K>0$ such that for all $q'\le q$, if $C_n\in \P_n$ has $\sigma^n(C_n) =[q']$, then $m_\phi(C_n)\le K\lambda^{-n}$;

\end{itemize}

Finally we might try to get contraction like in (CI), but not assuming that the orbits have to both start and end in a compact part $[\le q]$: for example only end in  $[\le q]$ which is related to condition (\hyperref[A]{A}) above, or only start in $[\le q]$, which is related to condition (\hyperref[B]{B}) above.

The above conditions can be shown to imply (\hyperref[UCS]{UCS}), but we show in Section~\ref{ssec:UCSweak} that they are strictly stronger, i.e., we have examples which satisfy (\hyperref[UCS]{UCS}), but fail all of these conditions.  We also note that the conditions of \cite{LivSauVai98} may fail, but  (\hyperref[UCS]{UCS}) hold.

\section{Proof of Theorem \ref{thm:contract}}
\label{sec:thmcontr}

Throughout this paper, since we are always assuming $P(\phi)<\infty$, by subtracting a constant from $\phi$ we can assume without loss of generality that $P(\phi)=0$.  Moreover, from here on, unless stated otherwise, we assume our alphabet $\A$ is $\N$.

Given $a, b\in \N$, define 
$$\ell(a,b):=\min\left\{k:\exists w\in\Sigma_k,[w]_0=a\text{ and } wb\in \Sigma_{k+1}\right\}.$$
This is finite by topological transitivity.
Given $q\in \N$, let
$$\ell(q):=\sup_{a, b\le q}\ell(a, b),$$
and note that this is also finite.

By topological transitivity, for any $x\in \Sigma$ and $n\in \N$, there is $\ell=\ell(x_n,x_0)\geq0$ and we can pick word $w(x_n,x_0)\in[x_n]$ of length $\ell$ such that the following concatenation is allowed
\begin{equation}
{w}(x_n,x_0)x_0, \dots,x_{n-1}\in \Sigma_{n+\ell},\label{eq:perlink}
\end{equation}
which are the first $n+\ell$ symbols of a periodic point $z=\left(\overline{{w}(x_n,x_0)x_0,\ldots, x_{n-1}}\right)$.    

\begin{lemma}
Suppose that $\phi$ has summable variations. Defining 
$$ \underline{C}_q(\phi):=\min_{a,b\leq q}\inf_{y\in[w(a,b)b]}\left\{S_{\ell(a, b)}\phi(y)\right\},$$
we have $ \underline{C}_q(\phi)>-\infty$.  
\label{lem:finitelink}
\end{lemma}

\begin{proof}
This follows since there are finitely many word pairs $a, b\le q$ to consider and since also $var_2(\phi)<\infty$, this means that $S_{\ell(a, b)}\phi$ is bounded on each  $[{w}(a,b)b]$.
 \end{proof}

\begin{lemma} For a topologically transitive CMS, $\phi$ of summable variations, (\hyperref[UCS]{UCS}) implies (\hyperref[CRC]{CRC}).
\label{lem:UCS-CRC}
\end{lemma}

\begin{proof}
    Let $q\in\N$ be given and pick $\lambda_q>0$ small enough such that $\chi_{per}(\phi)<-2\lambda_q<0$.
    Suppose $x\in\Sigma$ has $x_0, x_n\le q$ for $n\ge 1$.    
    Then for $z=\left(\overline{{w}(x_n,x_0)x_0,\ldots, x_{n-1}}\right)$ defined as above,
    \begin{align*}
    \underline C_q(\phi)+S_n\phi(x)&\leq S_{n+\ell}\phi(z)+2B_\phi<2B_\phi-\left(n+\min_{a,b\le q}\ell(a,b)\right)\lambda_q,
    \end{align*}
    where $\underline C_q(\phi)>-\infty$ as in \Cref{lem:finitelink}.  Hence $S_n\phi(x)<C_q-n\lambda_q$  where $C_q=\max\left\{0,2B_\phi-\underline{C}_q(\phi)\right\}$.    
\end{proof}

It is immediate that 
    (\hyperref[CRC]{CRC}) implies (\hyperref[CI]{CI}) since (\hyperref[CRC]{CRC}) does not impose any conditions on $x_j$ for $j\in[1,n-1]$, so it only remains to show that (\hyperref[CI]{CI}) implies (\hyperref[UCS]{UCS}) to complete the proof of Theorem~\ref{thm:contract}. 

The following lemma gives us a way of going from proofs of the results in this paper for non-positive potentials to general potentials $\phi$.

\begin{lemma}
There exists $\rho:\Sigma\to \R$ bounded on each 1-cylinder such that for $\phi':=\phi+\log\rho-\log\rho\circ\sigma$, we have $\phi'\le 0$.  Moreover, $\phi'$ has summable variations (or is weakly H\"older) if $\phi$ has summable variations (or is weakly H\"older).
\end{lemma}

This is essentially \cite[Lemma 1]{Sar01}, but we sketch parts of the proof here for completeness.

    \begin{proof}
        If $\phi$ is recurrent, then $\rho$ is the eigenfunction of the transfer operator $\L_\phi$ associated to $\phi$. If $\phi$ is transient,  take $\rho=\sum_{n\ge1}\L_\phi^n\mathbbm{1}_{[a]}$.  The regularity follows as in \cite[Lemma 1]{Sar01}, though there the shift is assumed topologically mixing: $\rho$ remains finite and non-positive in both cases under topological transitivity. 
    \end{proof}

The previous lemma provides a convenient upper bound on the potential for our proofs.  We next show that $\phi'$ inherits 
 (\hyperref[CI]{CI}) from $\phi$.
 
\begin{lemma}\label{lem:phi'}
    If $\delta_{\phi,\infty}<0$, then $\delta_{\phi',\infty}<0$.
\end{lemma}
    
(Note that in general this means that if $\delta_{\phi,\infty}<P(\phi)$, then $\delta_{\phi',\infty}<P(\phi')= P(\phi)$.)
    
    \begin{proof}
        By definition, there exist $\eps>0$ and $N_\eps,M_\eps, q_\eps$ such that \begin{equation}\label{eq: CI for phi}
        z_{\phi,n}(M,q_\eps)<-2\eps
        \end{equation} 
        for all $n>N_\eps$ and $M>M_\eps$.  Then for every $n>N_\eps$ large enough that $ \frac{n+\ell(q_\eps)}{2M}<\frac{n}{M}$, for every $x\in B$ for some $B\in B(n,2M,q_\eps)$, as in \eqref{eq:perlink}, there exists an admissible word ${w}={w}(x_{n-1},x_0)$, and a periodic point $y$ of period $n'=n+|{w}|$, such that $[y_0,\dots,y_{n+|w|-1}]\in B(n',M,q_\eps)$, $y=\left(\overline{x_0,\dots,x_{n-1}{w}(x_{n-1},x_0)}\right)$  and by summable variations, 
        $$S_n\phi'(x)\le S_{n'}\phi'(y)-\underline{C}_q(\phi')+B_{\phi'}=S_{n'}\phi(y)-\underline{C}_q(\phi')+B_{\phi'},$$
        where $\underline{C}_q(\phi')$ is defined as in \Cref{lem:finitelink}. Then \eqref{eq: CI for phi} implies
        $$S_n\phi'(x)<-2n'\eps-C_q(\phi')+B_{\phi'}$$ and by choosing $n$ large, this implies that for all $M>M_\eps$, $z_{\phi',n'}(M,q_\eps)<-\eps$, and consequently $\delta_{\phi',\infty}(q_\eps)<-\eps$. Since this inequality holds for all $q>q_\eps$, we conclude that $\delta_{\phi',\infty}<0$. 
        \end{proof}

Note that this lemma also holds for any $\psi=\phi+ \xi-\xi\circ\sigma$, provided $\xi$ has summable variations.  We will later use the following lemma to show that we cannot have a sequence of periodic measures whose integrals of $\phi$ converge to zero which moreover converge to a probability measure.

\begin{lemma}\label{lem:noeq}
There is no $q\in \N$ and sequence $(x^k)_k$ of periodic points of period $p_k$ such that $\frac1{p_k}S_{p_k}\phi(x^k)\to 0$  and $\nu_k([\le q]) \to 1$ as $k\to \infty$, where $\nu_k =\frac1{p_k}\sum_{i=0}^{p_k-1}\delta_{\sigma^ix^k}$.
\end{lemma}

In the proof we make use of the notion of a sequence of measures $(\mu_n)_n$ \emph{converging on cylinders} to a measure $\mu$: this means that for any $C\in \C_k$, $\mu_n(C) \to \mu(C)$ as $n\to \infty$, see \cite{IomVel21} for more details.

\begin{proof}
We will show that if we assume that the lemma is false then there is an equilibrium state $\nu$ having zero measure-theoretic entropy, which by for example \cite[Theorem 5.6]{Sar15} is a contradiction.

Let $\phi'$ be as in \Cref{lem:phi'}, and it is easy to see $S_n\phi(x)=S_n\phi'(x)$ for all $n\in\N$ and all $\sigma^nx=x$. Suppose there is such a $q$ and sequence of periodic points as in the lemma.  Note that $\int\phi'~d\nu_k\to 0$.  By \cite[Theorem 1.2]{IomVel21}, $\mathcal{M}_{\le1}(\Sigma)$ the space of shift-invariant sub-probability measures on $\Sigma$ is compact with respect to the convergence on cylinders topology, i.e., there is $\nu\in \mathcal{M}_{\le1}(\Sigma)$ such that $\nu_k\to\nu$ (up to subsequences) on cylinders.  Our assumption implies that $\nu$ is a probability measure.   Hence, as $(\nu_k)_k$, $\nu$ are probability measures, \cite[Lemma 3.17]{IomVel21} implies that the convergence also holds in the weak-* topology. In particular if $\phi'_L(x) = \max\{\phi'(x), -L\}$ then $\phi'_L$ is continuous and bounded and $\int\phi_L'~d\nu_{n_k}\to \int\phi_L'~d\nu$. 

\begin{claim*} Given $L>0$, for any $\eps>0$ there exists $K_\eps'$ such that $k\ge K_\eps'$ implies
$$\left|\int\phi_L'~d\nu_{n_k}- \int\phi'~d\nu_{n_k}\right|< \eps/4.$$
\end{claim*}

            \begin{proof}[Proof of Claim.]
            Since 
            $$ \int\phi'~d\nu_{n_k}= \int\phi_L'~d\nu_{n_k}+ \int_{\{\phi'<-L\}}\phi'~d\nu_{n_k},$$
            if the claim is false then there is $\eps>0$ such that for any $N\in \N$ we can find $k\ge N$ such that $\left|\int_{\{\phi'<-L\}}\phi'~d\nu_{n_k}\right|\ge \eps/4$, i.e., $\int_{\{\phi'<-L\}}\phi'~d\nu_{n_k}\le -\eps/4$.  But since $\phi'\le 0$, this means $ \int\phi'~d\nu_{n_k}\le -\eps/4$, contradicting the fact that $\int\phi'~d\nu_{n_k} \to 0$.
            \end{proof}
            
            Now given $L>0$, take $K_\eps\ge K_\eps'$, for $K_\eps'$ as in the claim and such that 
            $\left|\int\phi'~d\nu_{n_k}\right|< \eps/4$ and $\left|\int\phi_L'~d\nu- \int\phi_L'~d\nu_{n_k}\right|< \eps/2$ whenever $k\ge K_\eps$.  Then
            \begin{align}
            \begin{split}
            \left|\int\phi_L'~d\nu\right| &\le \left|\int\phi_L'~d\nu- \int\phi_L'~d\nu_{n_k}\right|+  \left|\int\phi_L'~d\nu_{n_k}\right|\\
            &<\frac{\eps}{2}+\left|\int\phi'_L~d\nu_{n_k}-\int\phi'~d\nu_{n_k}\right|+\left|\int\phi'~d\nu_{n_k}\right|\\
            &<\frac{\eps}{2}+\frac{\eps}{4}+\frac{\eps}{4}=\eps.\label{eq:phiLbd}
            \end{split}
            \end{align}
Now the Monotone Convergence Theorem implies $-\int\phi_L'~d\nu \nearrow -\int\phi'~d\nu= -\int\phi~d\nu$ as $L\to \infty$.  Moreover \eqref{eq:phiLbd} and weak* convergence of $\nu_{n_k}$ to $\nu$ imply $\left|\int\phi~d\nu\right|< \eps$ for all $\eps$, i.e., $\int\phi~d\nu=0$.  Since $h(\nu)+ \int\phi~d\nu\le P(\phi)=0$, this is a contradiction.
\end{proof}

\begin{proposition}\label{prop:CI-UCS}
    Under the assumptions of Theorem~\ref{thm:contract}, (\hyperref[CI]{CI}) implies (\hyperref[UCS]{UCS}).
\end{proposition}

The idea of the proof is that  (\hyperref[UCS]{UCS}) must hold for orbits which `spend most of their time in a compact part' of the space, something which holds in the finite alphabet case but here is assured by Lemma~\ref{lem:noeq}, and then (\hyperref[CI]{CI}) ensures that orbits which `spend significant time outside the compact part' satisfy contraction and hence (\hyperref[UCS]{UCS}).

\begin{proof}
Suppose (\hyperref[UCS]{UCS}) fails. Then there exists a sequence of periodic points $x^1,x^2,\dots$ with periods $p_1,p_2,\dots$ and with Birkhoff averages $$s_n=\frac{1}{p_k}S_{p_k}\phi(x^n)\le0$$ for all $n$ and $\lim_{n\to \infty}s_n=0$.

By the definition of (\hyperref[CI]{CI}) and \Cref{lem:phi'}, for all $\eps>0$ such that $\delta_{\phi',\infty}<-\eps<0$, there exist $N_\eps$, $M_\eps,q_\eps$ such that for $n>N_\eps$, $M>M_\eps$, $q>q_\eps$ and all $x$ such that $[x_0,\dots,x_{n-1}]\in B(n,M,q)$, $S_n\phi'(x)<-n\eps$.

Given $q, N\in \N$, let $\A_{[\le q], N}$ be the set of words $w\in \Sigma_N$ such that $[w]_0\le q$ and $[w]_i> q$ for $i=1, \ldots, N$ and $wq'\in \Sigma_{N+1}$ for some $q'\le q$, i.e., the \emph{first return words to $[\le q]$ of length $N$}.  Let $\A_{[\le q]}:= \cup_{N\ge 1}\A_{[\le q], N}$ and \[[w,\le q]:=\bigcup_{q'\le q}\{[wq']:wq'\in\Sigma_{|w|+1}\}\]

Given $x\in \Sigma$ such that $x_0, x_n\le q$ we can decompose $x_0, \ldots, x_{n-1} = w_1 v_1w_2 v_2 \ldots w_kv_k$, where $w_i\in \A_{[\le q]}$ and ${v}_i\in \{\emptyset\}\cup\left(\cup_m\{1, \ldots, q\}^m\right)$ for all $1\le i\le k$.  For each $x$, let $\mathcal D(x,q)$ denote the set of words $w_i$ in this decomposition. 

Given $q>q_\eps$, define the proportion function $\zeta(\cdot)$ 
\begin{equation}\label{eq:proportion_fn}
    \zeta(x^n):=\frac{1}{p_n}\sideset{}{'}\sum_{\left\{{w}\in \mathcal D(x^n
 ,q_\eps)\cap (\bigcup_{N\ge N_\eps}\mathcal A_{\le q_\eps,N})\right\}}|{w}|.
\end{equation}
Here $\sum'$ means that we count with multiplicity, i.e., if $w$ appears $k$ times in the decomposition of $x^n$ we sum $|w|$ $k$ times.

Notice that since $w_j\le q$ if and only if $j=0$ for each ${w}\in \A_{[\le q], N}$, so long as $(N+1)M_\eps\ge1$, $S_N\phi'(x)<-N\eps$ for $x\in  [{w}]$.  Hence, since we can assume that $(N_\eps+1) M_\eps\ge 1$,  we can show that $\limsup_{n\to \infty}\zeta(x^n)=0$: if there exists $\eta>0$ such that $\lim_{n\to \infty}\zeta(x^n)\ge \eta$, by the non-positivity of $\phi'$, 
    \begin{align}
    \begin{split}
    \liminf_{n\rightarrow\infty}s_n&\le\liminf_{n\rightarrow\infty}\frac{1}{p_n}\sideset{}{'}\sum_{\left\{{w}\in \mathcal D(x^n
 ,q_\eps)\cap (\bigcup_{N\ge N_\eps}\mathcal A_{\le q_\eps,N})\right\}}\sup_{x\in[{w},\le q_\eps]}S_{|{w}|}\phi'(x)\\
    &\le\liminf_{n\rightarrow\infty}\frac{1}{p_n}\sideset{}{'}\sum_{\left\{{w}\in \mathcal D(x^n
 ,q_\eps)\cap (\bigcup_{N\ge N_\eps}\mathcal A_{\le q_\eps,N})\right\}}-|w|\eps\\
        &\le\limsup_{n\rightarrow\infty} -\zeta\left(x^n\right)\eps\le-\eta \eps<0, \label{eq:proportion}
        \end{split}
        \end{align}
contradicting our assumption that $\lim_{n\to\infty}s_n=0$.        
        
        Next since by the $\F$-property and topological transitivity, $\#\{\cup_{N\le N_\eps} \A_{[\le q], N}\} <\infty$ and the following quantity is also finite:
    \begin{equation}
                q'(q,N):=\min\left\{q'\in\N: \text{if } {w}\in\cup_{N\le N_\eps} \A_{[\le q], N} \text{ then } w\in [\le q']^{|w|}\right\}. 
            \end{equation}
Therefore, since $\limsup_{n\to \infty}\zeta(x^n)=0$,  $$\lim_{n\to \infty}\nu_n\left([\le q']\right)=1 \text{ where }\nu_k=\frac{1}{p_k}\sum_{i=j}^{p_k-1}\delta_{\sigma^ix^k},$$ which this contradicts \Cref{lem:noeq}, hence (\hyperref[CI]{CI}) implies (\hyperref[UCS]{UCS}).
\end{proof}

\section{Proof of Theorem~\ref{thm:UCS-SPR}}
\label{sec:UCS-SPR}

Assuming now (\hyperref[UCS]{UCS}) holds, then (\hyperref[CI]{CI}) follows from \Cref{thm:contract} and in the next theorem we then show the system is (\hyperref[def:SPR]{SPR}) under slightly more general assumptions than those in \Cref{thm:UCS-SPR}.

\begin{theorem}
\label{thm:CI-SPR}
    Let $(\Sigma,\phi) $ be a topologically transitive CMS with $h_{top}<\infty$.  Suppose that $\phi$ a potential of summable variations with $P(\phi)<\infty$, satisfying (\hyperref[CI]{CI}) and assume that $\delta_{\phi,\infty}+h_{\infty}<P(\phi)$. Then $\phi$ is (\hyperref[def:SPR]{SPR}).
\end{theorem}

We provide an example in \Cref{ssec:eg_not_spr} with $\delta_{\phi,\infty}+h_{\infty}=P(\phi)$ but the systems fails to be SPR, so the condition $\delta_{\phi,\infty}+h_{\infty}<P(\phi)$ above is sharp.

\begin{lemma}\label{lem:few_rets}
 For all $\eps>0$, there exists $q$ and $K_q\ge0$ such that for all positive $n$, if $x_0,x_n\le q$ and $x_k>q$ whenever $1\le k\le n-1$, then $$S_n\phi(x)<K_q+n\left(\delta_{\phi,\infty}+\eps\right).$$
\end{lemma}

\begin{proof}
Let $\eps>0$ be given.  By definition there exists $q$ such that $\delta_{\phi,\infty}(q,M)<\delta_{\phi,\infty}+\frac{\eps}{2}$ for all $M$ large. Then there exists $N_\eps$ such that for all $n>N_\eps$ if $x\in \Sigma$ is such that $x_0,x_n\le q$, but $x_k>q$ for $i=1,\dots,n-1$, then
    $$\frac{1}{n}S_n\phi(x)<\delta_{\phi,\infty}(q,M)+\frac{\eps}{2}<\delta_{\phi,\infty}+\eps.$$
    Since the $\F$-property implies that for each $n$ the number of words of length $n$ which start and end at $[\le q]$  are finite, also using summable variations, $$K_q:=\max\left\{\max_{n\le N_\eps}\sup\left\{S_n\phi(x):x_0,x_n\le q\right\},0\right\}$$
    is finite and satisfies the lemma.
\end{proof}

Given $q\in \N$ as in the lemma, let $Y=[\le q]$ and define $\tau_Y: Y\to\N$ by $\tau_Y(x):=\inf\{n\ge 1: \sigma^n(x)\in Y\}$.  Then let $F:Y\to Y$ be the first return map $F=\sigma^{\tau_Y}$.  Let $\C^F$ denote the collection of 1-cylinders for $(Y, F)$, so that $Z\in \C^F$ implies that $F(Z)=[a]$ for some $a\le q$.  The topological transitivity of the original system means that there is some $J\in \N$ such that for any $Z, Z'\in \C^F$ there is $j\le J$ such that $Z'\subset F^j(Z)$, which is a stronger condition than the BIP property (see \cite{Sar03}).

Define the corresponding induced potential $\hat{\phi}_Y=\sum_{i=0}^{\tau_Y-1}\phi\circ\sigma^i$ and note that this has summable variations (in fact $var_1\hat{\phi}_Y<\infty$), so $B_{\hat\phi}<\infty$.  By for example \cite[Theorem 2]{Sar01}, $P(\hat\phi)\le 0$, so setting $\overline\phi=\hat\phi-P(\hat\phi)$ we have a normalised potential, \ie of zero pressure, and there is an $\overline\phi$-conformal measure $m_Y$ and an equivalent invariant measure $\mu_Y$, see \cite[Theorem 1]{Sar03}; also $B_{\overline\phi}=B_{\hat\phi}<\infty$.  Note that if $\phi$ is recurrent then $\hat\phi=\overline\phi$.  We also define $S^{F}_n\overline{\phi}_Y=\sum_{i=0}^{n-1}\overline{\phi}_Y\circ F^i.$  

\begin{lemma}\label{lem:meas_few_rets}
There is $C_1>0$ and $\eps>0$ such that if $Z\in \C_p^F$ and for some $p\ge1$, $$\sum_{i=0}^{p-1} \tau_Y(F^i(Z))=n$$ then
$$m_{Y}(Z)\leq C_1^p\exp{\left(n(\delta_{\phi,\infty}+\eps)\right)}.$$
\end{lemma}

\begin{proof}
Writing $\tau_Y(F^i(Z)) =\tau_i$,  $Z$ is an $(n+1)$-cylinder with respect to $\sigma$ of the form 
$$[z_0,\dots,z_{\tau_1-1}, z_{\tau_1},\dots,z_{\tau_p-1},z_{\tau_p}]$$
where $z_0,z_{\tau_i}\leq q$ for $i=1,\dots,p$. By conformality and \Cref{lem:few_rets},
\begin{align*}    m_{Y}\left([z_0]\right)&=\int_{[z_0,\dots,z_{\tau_p-1}]}e^{-S_n\phi(x)+pP(\hat\phi)}\,dm_{Y}\\
& \geq m_{Y}([z_0,\dots,z_{\tau_p-1}])e^{-\sup_{x\in Z} S_n\phi(x)+pP(\hat\phi)}.
\end{align*}
Hence   $$ m_{Y}([z_0,\dots,z_{\tau_p-1}])\leq \exp\left(p(K_q-P(\hat\phi))+n(\delta_{\phi,\infty}+\eps)\right)m_Y([z_0]),$$
    so setting $C_1=e^{K_q-P(\hat\phi)}$, we are finished.
  \end{proof}

\begin{proof}[Proof of Theorem~\ref{thm:CI-SPR}]
The proof is similar to that of \cite[Theorem 7.14]{DobTod23}.  As we will see, by \eqref{eq:SPRequiv} it suffices to show the inducing scheme on some 1-cylinder $[a]$, that is the first return map $([a], \sigma^{\varphi_a})$ to $[a]$,  has an exponential tail.

    Pick $\eps>0$ such that 
    $$\delta_{\phi,\infty}+h_{\infty}<-4\eps,$$
    choose $q$ satisfying \Cref{lem:few_rets} and such that for all large $M$, \begin{equation}\label{eq:delta+h<eps}
        \delta_{\phi,\infty}+h_{\infty}(M,q)<-3\eps.
    \end{equation} 
    
    So for $(Y, F)$ as above, which must also satisfy \Cref{lem:meas_few_rets}, by topological transitivity there exists $N$ such that for all $Z\in \C_1^F$,
    \begin{equation}\label{eq:Y cover}Y\subset\bigcup_{j=1}^{N}F^{j}(Z).\end{equation}
Pick some 1-cylinder with respect to $\sigma$, $Y_0=[a]$, with $m_{Y}(Y_0)>0$ and let $m_{Y_0}$ be the conditional conformal measure here.

\begin{claim} There is some uniform constant $\beta>0$ such that for $Z_n\in\C_n^F$,
\begin{equation*}
\frac{m_{Y}\left(x\in Z_n: F^j(x) \notin Y_0, j=n, \ldots, n+N-1\right)}{m_{Y}(Z_n)}< e^{-\beta}.\label{eq:induction}
\end{equation*}
\label{claim:avoid}
\end{claim}

\begin{proof}[Proof of \Cref{claim:avoid}.] 
By \eqref{eq:Y cover}, for each $b\in\mathcal A$ such that $[b]\subset Y$ there is some cylinder (with respect to $F$) $A\subset [b]$  and $0\le k(A)\le N-1$ such that $F^{k(A)}(A) = [a]$. Denote the (finite) collection of such cylinders by $\B$.  In particular there is some $A\in \B$ such that $A\subset F^n(Z_n)$.   Letting  $A'=F^{-n}A\cap Z_n$, and it suffices to find a lower bound for $\frac{m_Y(A')}{m_Y(Z_n)}$, independent of $Z_n\in \C_n^F$ and $A\in \B$.

Then
 $\min_{A\in \B}\inf_{x\in A}S_{k(A)}^F\,\overline{\phi}_Y(x)$ is bounded from below by the finiteness of $\B$ and summable variations (similarly to \Cref{lem:finitelink}). 
 
By conformality, for any $C\subseteq Y$, if $F^m:C\to F^mC$ is injective, $m_Y(F^mC)=\int_C\exp(-S_m^F\,\overline{\phi}_Y)dm_Y$, hence
\begin{align*}
    \frac{m_Y(A')}{m_Y(Z_n)}&\ge\frac{m_Y(A)}{m_Y(F^{n}Z_n)}\exp{\left(-\sup_{Z_n}S_{n}^F\,\overline{\phi}_Y+\inf_{A'}S_{n}^F\,\overline{\phi}_Y\right)}\\
    &\ge\frac{m_Y([a])}{m_Y(Y)}e^{-B(\overline{\phi}_Y)}e^{\inf_{x\in A}S^F_{k(A)}\,\overline{\phi}_Y(x)}>0,
\end{align*}
uniformly, as required.
\end{proof}

\begin{claim} For each $k\ge1$,
$$\frac{m_{Y}\left(x\in Y_0: F^j(x) \notin Y_0, j=1, \ldots, kN\right)}{m_{Y}(Y_0)} < e^{-k\beta}.$$
\label{claim:multavoid}
\end{claim}

\begin{proof}[Proof of \Cref{claim:multavoid}.]
This claim is proved by induction. As $Y_0=[a]$ can be written as a union of $1$-cylinders with respect to $F$, \Cref{claim:avoid} and the fact that for all positive numbers $a,b,c,d$, $\frac{a+c}{b+d}\le\max\left\{\frac ab,\frac cd\right\}$, together implies
$$\frac{m_{Y}\left(x\in Y_0: F^j(x) \notin Y_0, j=1, \ldots, N\right)}{m_{Y}(Y_0)} < e^{-\beta}.$$
Assume inductively that for each $i\ge1$,
$$\frac{m_{Y}\left(x\in Y_0: F^j(x) \notin Y_0, j=1, \ldots, iN\right)}{m_{Y}(Y_0)} < e^{-i\beta}.$$
Defining the set 
$$\mathcal{Z}_i:=\left\{Z\in \mathcal{C}^F_{iN+1}:Z\subset Y_0,\, F^j(Z) \notin Y_0, \,j=1, \ldots, iN\right\},$$ by \Cref{claim:avoid} and the inequality above:
\begin{align*}
& \frac{m_{Y}\left(x\in Y_0: F^j(x) \notin Y_0, j=1, \ldots, (i+1)N\right)}{m_{Y}(Y_0)}
 \\
&= \frac1{m_{Y}(Y_0)} \sum_{Z\in \mathcal{Z}_i} m_{Y}(Z) \frac{m_{Y}\left(x\in Z: F^j(x) \notin Y_0, j=iN+1, \ldots, (i+1)N\right)}{m_{Y}(Z_{i(N+1)})}\\
& < \frac1{m_{Y}(Y_0)}  \sum_{Z\in\mathcal{Z}_i} m_{Y}(Z) e^{-\beta}\leq e^{-\beta}\frac{m_{Y}\left(x\in Y_0:F^j(x)\notin Y_0,\,j=1,\dots,iN\right)}{m_{Y}(Y_0)}\\
&< e^{-(i+1)\beta}.\qedhere\end{align*}
\end{proof}

Letting $T=\gamma n$ for some $\gamma\in (0, 1)$ to be determined later, we can split the set $\left\{x\in Y_0:\varphi_a(x)=n\right\}$ depending on whether $x$ visits $Y$ more or less than $T$ times in its first $n$ symbols, which can be written
    \begin{align*}
        m_{Y_0}\left(\{\varphi_a=n\}\right)&\le \,m_{Y_0}\left(\left\{\varphi_a(x)=n,\sum_{j=0}^{T}\tau_Y(F^j(x))>n\right\}\right)\\&\quad+m_{Y_0}\left(\left\{\varphi_a(x)=n,\sum_{j=0}^{T}\tau_Y(F^j(x))\le n\right\}\right)\\
        &=: I+II.
    \end{align*}
    By \Cref{claim:multavoid},
    \begin{align*}
    I\le \sum_{p=T}^n\exp{\left(-\frac{p}{N}\beta\right)}\le C_2\exp{\left(-\frac{T}{N}\beta\right)},
    \end{align*}
    for some $C_2\in\R$. The number of $n$-cylinders with respect to $\sigma$ which spend a proportion $\gamma\le 1/M$ of their $\sigma$-iterates up to $n$ in $Y$ is no more than $\#B(n,M,q)$.  Moreover, for all large $n$, $\#B(n,M,q)\le C_3e^{n(h_\infty(M,q)+\eps)}$ for some $C_3>0$, so combining this with \Cref{lem:meas_few_rets} and \eqref{eq:delta+h<eps} we get 
    \begin{align*}
    II&\le C_1^T\exp\left(n(\delta_{\phi,\infty}+\eps)\right)\#B(n,M,q)\\
    &\le C_1^TC_3\exp{\left(n(\delta_{\phi,\infty}+h_{\infty}(M,q)+2\eps)\right)}\le C_1^TC_3\exp{(-n\eps)}.
    \end{align*}
Then choosing $\gamma=\min\left\{\frac{1}{M},\frac{\eps}{2\log{C_1}}\right\}$, both $I$ and $II$ are exponentially small so that 
     $$\limsup_{n\to\infty}\frac{1}{n}\log{m_{Y_0}(\{\varphi_a=n\})}<0.$$ 
     As $m_{Y_0}$ is conformal, 
     $$m_{Y_0}(\{\varphi_a=n\})\asymp \sum_{\sigma^nx=x, \varphi_a(x)=n}e^{S_n\phi(x)-i(x)P(\hat\phi)} \ge Z_n^*(\phi, a),$$
     where $i(x)$ corresponds to the number of hits to $Y$ before $Y_0$.  Hence  \eqref{eq:SPRequiv} holds and the system is strongly positive recurrent. 
\end{proof}

\Cref{thm:CI-SPR} means that  (\hyperref[UCS]{UCS}) implies (\hyperref[def:SPR]{SPR}).
For the other direction of \Cref{thm:UCS-SPR}, it suffices to prove the statement under topological mixing since we can use \emph{spectral decomposition}, a tool to reduce arguments on topologically transitive to topologically mixing. Briefly speaking, if $(\Sigma,\sigma)$ is a topologically transitive CMS with period $p$, the alphabet is divided into $p$ equivalence classes $\{\mathcal{A}_1\dots,\mathcal{A}_{p-1}\}$ and $\Sigma=\biguplus_{i=0}^{p-1}\Sigma_i$, $\Sigma_i=\left\{x\in\Sigma:x_0\in\mathcal{A}_i\right\}.$ Then $(\Sigma_i,\phi_p,\sigma^p)$ is conjugate to a topological mixing CMS and most statements (especially those in this paper) proved for $(\Sigma_i,\phi_p,\sigma^p)$ remain valid for the original CMS. For more detailed discussion, see for example \cite[\S2.2,\S6]{RuhSar22}.

\begin{proposition}Under the assumptions of Theorem~\ref{thm:UCS-SPR}, (\hyperref[def:SPR]{SPR}) implies (\hyperref[UCS]{UCS}).\label{SPR implies UCS}
\end{proposition}

    \begin{proof}
         By (\hyperref[def:SPR]{SPR}) there is  $a\in\mathcal{A}$ such that $\Delta_a[\phi]>0$. First by \cite[Lemma 3]{Sar01}, $P(\phi)=0$ implies the induced pressure on $[a]$, $P(\overline{\phi})$, is zero, and (\hyperref[def:SPR]{SPR}) implies that there exists $\eps_a>0$ such that 
        \begin{equation}
        \label{eq:2eps<infty}P(\overline{\phi+2\eps_a})<\infty.
        \end{equation}
 Moreover, as in \eqref{eq:SPRequiv}, there exists $N_a\in\N$ such that for all $n>N_a$, all $x$ such that $\varphi_a(x)=n$, 
 \begin{equation}
 \frac{1}{n}S_n\phi(x)<-\eps_a.
 \label{eq:SPRcont}
 \end{equation}       
        
        Suppose by contradiction that $\chi_{per}(\phi)= 0$; take $\phi'$ as in \Cref{lem:phi'} and $\chi_{per}(\phi')=\chi_{per}(\phi)$. Then there exists a sequence of periodic points $x^1,x^2,\dots,$ with period $p_1,p_2,\dots$ and Birkhoff averages 
        \begin{equation}\label{eq:BA lower bounds}
        s_n=\frac{1}{p_n}S_{p_n}\phi(x^n)=\frac{1}{p_n}S_{p_n}\phi'(x^n)>-\eps_a \text{ and } \lim_{n\to \infty}s_n=0.
        \end{equation}  
         
\textbf{Case 1.} Suppose there exists $x\in\left\{x^1,\dots\right\}$ such that $\forall k\ge 0$, $x_k\ne a$.  Then as in \eqref{eq:perlink}, by topological transitivity, there are words ${v}$, ${w}$ of length $\ell_1=\ell(a,x_0)$ and $\ell_2=\ell(x_{n-1},a)$ respectively such that $[{v}]_0=a,$ $vx_0\in \Sigma_{|v|+1}$, $x_{n-1}w\in \Sigma_{|w|+1}$,  ${w}a\in \Sigma_{|w|+1}$, hence
        $${v}x_0,\dots,x_{n-1}{w}\in \Sigma_{\ell_1+n+\ell_2}.$$
 Moreover, for each $k\in\N$ and $n_k=kn+\ell_1+\ell_2$ there is a periodic point $y(k)\in[a]$ with $\varphi_a(y(k))=n_k$ of the form:
        $$y(k)=\left(\overline{{v}\left(x_0,\dots,x_{n-1}\right)^k{w}}\right)$$ where $\left(x_0,\dots,x_{n-1}\right)^k$ means the string is repeated $k$ times. By summable variations, there exists a constant $C>0$ such that for all $k$,
        \begin{equation*}
       Z_{n_k}^*(\phi,a)\ge \exp{\left(S_{n_k}\phi(y(k))\right)}\ge \exp{\left(C-kn\eps_a\right)}.
        \end{equation*}
        Then as in \cite[(5)]{Sar01}, $\left|P(\overline{\phi+p})-\log\sum_{k\ge1}e^{kp}Z_k^*(\phi,a)\right|\le B_{\phi}$, therefore,
        $$    
        \infty=\log{\sum_{k=1}^{\infty}e^{n_k\eps_a+C}e^{S_{n_k}\phi(y(k))}}\le C+\log\sum_{n=1}^{\infty}e^{n\eps_a}Z^*_n(\phi,a)\le P(\overline{\phi+\eps_a})+B_\phi+C,
        $$
        which is a contradiction to \eqref{eq:2eps<infty} since $B_\phi<\infty$.\\
        \textbf{Case 2.}
        Now suppose all $x\in\{x^1,\dots\}$ contain state $a$. Without loss of generality one can suppose $x^i_0=a$ for all $i$ by periodicity.   Let $\A_a\subset\Sigma^*$ be words $w$ where $[w]_i=a$ if and only if $i=0$, and moreover ${w}a\in \Sigma_{|w|+1}$, the set of \emph{first return words} to $a$.  \
        
        For all $n,$ 
        $$x^n=\left(\overline{{w}_0\dots{w}_{k_n-1}}\right) \text{ for some $k_n\ge1$, }{w}_i\in \A_a \text{ and } \sum_{j=0}^{k_n-1}|{w}_j|=p_n;$$
        That is, each $x^n$ can be decomposed into several first return words.
        
        As $S_m\phi(x)=S_m\phi'(x)$ for any periodic point with period $m$, \eqref{eq:SPRcont} implies that for all first return word $\alpha$ with length longer than $N_a$, 
        \begin{equation*}
       \sup_{x\in[\alpha a]}S_{|\alpha|}\phi'(x)<-|\alpha|\eps_a\label{eq:strictly negative word length}
        \end{equation*}
        Letting $\A_{a, >k}:= \{w\in \A_a: |w|>k\}$, re-define the proportion function similarly to \eqref{eq:proportion_fn},
         $$\zeta:\left\{x^1,\dots\right\}\rightarrow[0,1],$$$$\zeta(x^n)=\frac{1}{p_n}\sideset{}{'}\sum_{{w}\in\left\{ x^n_0\dots x^n_{p_n-1}\cap \A_{a, >N_a} \right\}}|{w}|,$$
          where $\Sigma'$ again means that we count with multiplicity.
        Then repeating \eqref{eq:proportion} with $\eps=\eps_a$ this definition ensures $\lim_{n\to \infty}\zeta(x^n)=0$ since otherwise we contradict the property \eqref{eq:BA lower bounds} of our periodic points.  By the $\F$-property, we can define the function $q_a$ by
                   \begin{equation*}
                q_a(N):=\min\left\{ q\in\N: \begin{split}&\text{if } {w}\in\mathcal{A}_a \text{ and } |{w}|\le N,\\[-1mm]
                 &\text{ then } [w]_i\le q \text{ for } i=0, \ldots, |w|-1\end{split} \right\}. 
            \end{equation*}
            The sequence of probability measures 
            $$\nu_n=\frac{1}{p_k}\sum_{j=0}^{p_n-1}\delta_{\sigma^jx^n}$$ 
            satisfies $\lim_{n}\nu_n([\le\! q_a(N_a)])=1.$ But since $\lim_{k\rightarrow\infty}s_{n_k}=0$, we have a contradiction to \Cref{lem:noeq}, hence such sequence of periodic points does not exist.
            \end{proof}

This concludes the proof of Theorem~\ref{thm:UCS-SPR}.

\section{Examples}
\label{sec:examples}

The conditions for our main results are weak, so there are many CMS satisfying these, but in this section we focus on a set of examples which simultaneously represent a broad class of CMS for which our theory applies as well as exhibiting edge cases which demonstrate the sharpness of our results.

Our examples take the form of `bouquet' Markov graphs, see  \cite{Rue03, Rue19}.  We note that we could have used other shift models, eg $S$-gap shifts, to demonstrate the sharpness of our results, but the bouquet setup covers these as well as all other topologically transitive CMS, as we note at the end of Section~\ref{ssec:bouqind}.

Some of the examples which inspired this work come from codings for dynamical systems, particularly in the case of  interval maps $f:[0, 1]\to [0, 1]$: here a subset $Y\subset [0, 1]$ is taken and some return time $\tau:Y\to \N\cup\{\infty\}$ is chosen so that $F= f^\tau$, an \emph{inducing scheme} defines a Markov map on $Y$, i.e., there is a partition $\{Y_i\}_i$ such that $\tau|_{Y_i}$ is some constant $\tau_i$ and $F(Y_i)$ is a collection of elements of this partition.  We can associate bouquet Markov graph to this, described in more detail below, which then defines a shift, and we can take the potential $\phi:I\to [-\infty, \infty]$ and lift this to the symbolic model.  For example, such a coding can be done for general multimodal maps of the interval, as shown in, for example, \cite[Theorem 3]{BruTod09}, or more classical and specific inducing schemes like those given in \cite{BruLuzStr03} (which include Collet-Eckmann maps).  Note that our framework here does not give new results in the specific setups in the references mentioned.

\subsection{Bouquet setup}
 
Following \cite{Rue03, Rue19}, let $a:\N\to \N_0$ with $a(1)=1$.   We define our set of vertices as 
$$V:= \{r\} \cup \bigcup_{n=1}^\infty\left\{ v_k^{n, i}: 1\le i\le a(n), \ 1\le k\le n-1\right\},$$
where all vertices with distinct labels above are distinct vertices.  We call $r$ the \emph{root}. For notational convenience write $v_0^{n, i}= v_n^{n, i} =r$.   
Then the only allowed transitions in our Markov graph are  $v_k^{n, i} \mapsto v_{k+1}^{n, i}$ for $0\le k\le n-1$.   This defines a \emph{bouquet} of loops: with $a(n)$ disjoint \emph{simple} loops (from $r$ back to $r$) of length $n$.  The resulting shift space which we refer to as a \emph{bouquet shift} is $\Sigma=\Sigma_V$: it has $a(n)$ periodic cycles of prime period $n$.

Below we will make various choices of $(a(n))_n$ and potentials $\phi:\Sigma_V\to \R$.  Our analysis will be via first returns to $[r]$.  Note that \cite{Rue03, Rue19} were concerned with measures of maximal entropy (in which case we set $\phi\equiv -h_{top}(\sigma)$), rather than the more general setting of equilibrium states that we are interested in here.

For various calculations in this section we must replace $V$ with $\N$.  This can be done by listing the elements of $V$ in the natural order, starting with $r$.  Before moving to more specific examples we prove a useful lemma, which applies to all bouquet shifts.

\begin{lemma}
$h_\infty=\limsup_{n\to \infty} \frac1n\log a(n) $.
\label{lem:hinfa(n)}
\end{lemma}

\begin{proof}
Suppose that $\limsup_{n\to \infty} \frac1n\log a(n) =\log\lambda$, which we may assume is finite, as otherwise the conclusion is immediate.   Then for $\eps>0$ there is $C>0$ such that for an infinite sequence of $n\in\N$,
$$ \frac1C\lambda^{n(1-\eps)}\le a(n) \le C\lambda^{n(1+\eps)},$$
and indeed the upper bound holds for all $n\in \N$.

We first show that $h_\infty\ge \limsup_{n\to \infty} \frac1n\log a(n)$. Notice that for $M, q\in \N$, if $n$ is large enough so that $(n+1)M\ge1$ and so that any of the simple loops of length $n$ only intersect $[\le q]$ at the root $1$ (for example if $n>n_q$ where $q\le \sum_{k=1}^{n_q}a(k)$), we have $z_n(M, q) \ge  a(n)\ge \frac1C\lambda^{n(1-\eps)}$ for an infinite sequence of such $n$, hence the required lower bound holds.

For the upper bound, let
$$\P(q):=\left\{u\in \Sigma^*: [u]_i\le q \text{ for }i=0, \ldots |u|-1\text{ and } u1\in \Sigma_{|u|+1}\right\},$$
$$\mathcal{G}(q):=\left\{v\in \Sigma^*: [v]_0, [v]_{|v|-1}=1 \right\}.$$
$$\mathcal{S}(q):=\left\{w\in \Sigma^*: [w]_i\le q \text{ for }i=0, \ldots |w|-1\text{ and } 1w\in \Sigma_{|w|+1}\right\}$$ 
be the set of $q$-prefixes, $q$-good words, and $q$-suffixes respectively.  Then any cylinder $[x_0, \ldots,x_n]\in B(n, M, q)$ can be decomposed so that $x_0, \ldots,x_n = uvw$ where $u\in \P(q), v\in \mathcal{G}(q), w\in \mathcal{S}(q)$.

Then  
\begin{equation*}
z_n(M, q) \le \#\left\{uvw\in \Sigma_{n+1}: \begin{split} & u\in \P, v\in \mathcal{G}, w\in \mathcal{S} \\[-3mm]
&\text{ and } \#\{0\le i\le |v|-1: [v]_i=1\}\le \frac{|v|}M\end{split}\right\}
\end{equation*}

which can be bounded above by $\#\P(q)\#\mathcal{S}(q)\tilde z_n(M, q)$ where
\begin{align*}
\tilde z_n(M, q)&\le \sum_{k\le n/M}\sum_{i_1+ \cdots+i_k=n} a(i_1)\cdots a(i_k) \\
&\le C \sum_{k\le n/M}\sum_{i_1+ \cdots+i_k=n} \lambda^{n(1+\eps)}\le C2^{\frac nM}\lambda^{n(1+\eps)} =C\left(2^{\frac 1M}\lambda^{1+\eps}\right)^n,\end{align*}
so taking appropriate limits, and noting that $\#\P(q), \#\mathcal{S}(q)<\infty$, we obtain $h_\infty=\log \lambda$. 
\end{proof}

\subsection{(\hyperref[UCS]{UCS}) is a weak condition}
\label{ssec:UCSweak}

Here we will use a simple set of examples to compare (\hyperref[UCS]{UCS}) with other conditions of this type.

 Set $a(n) = 1$ for all $n$, $\phi|_{[rv_1^{n, 1}]}= -n\log 2$ and $\phi=0$ otherwise.  For the first return map to $[r]$ the induced potential $\overline{\phi}:[r]\to \R$, takes the value $ -n\log 2$ on the vertex corresponding to the loop of length $n$.  Then
 $$P(\overline{\phi}) = \log\left(\sum_{n\ge 1} \frac1{2^n}\right)=0.$$ 
Since, moreover, $\sum_{n\ge 1} \frac n{2^n}<\infty$, $\phi$ is positive recurrent and has $P(\phi)=0$. 

The system $(\Sigma_V, \sigma, \phi)$ clearly satisfies (\hyperref[UCS]{UCS}) since for the periodic point $x_n$ of prime period $n$, $\frac1nS_n\phi(x_n) = -\log 2$.  On the other hand, the hyperbolicity condition as in \cite{InoRiv12, LiRiv14a} fails since for any $n$, there is a point $y_n\in [v_1^{n, 1}v_2^{n, 1}\cdots v_{n-1}^{n, 1}]$ such that $S_n\phi(y_n)=0$.  Moreover, (\hyperref[A]{A}), (\hyperref[C]{C}) and (\hyperref[D]{D}) also fail for the same reasons, as well as a condition like (\hyperref[CI]{CI}) where orbits are not assumed to start in a compact part.  Finally, regarding the conditions of \cite{LivSauVai98}, this would require $\sum_{C\in \mathcal{C}_1}\sup_{x\in C} e^{\phi(x)}<\infty$ as well as a condition like hyperbolicity to hold, both of which fail here.

 We can modify the example to $\phi|_{[v_{n-1}^{n, 1}r]}= -n\log 2$ and 0 otherwise, to obtain the same induced system as above, but here we see that we can fail  (\hyperref[B]{B}), as well as  condition like (\hyperref[CI]{CI}) where orbits are not assumed to end in a compact part, whilst satisfying (\hyperref[UCS]{UCS}).  If we wished to fail all of these conditions apart from (\hyperref[UCS]{UCS}), for a loop of length $n$ we could put the weight halfway along the loop.

\begin{remark}
If we wanted $\phi$ to be uniformly bounded, then the above examples can be smoothed out, eg putting weight $-2\log 2$ on $n/2$ of the vertices in the loop of length $n$ (suitably adjusting for when $n$ is odd).
\label{rmk:spread}
\end{remark}

\subsection{Example the showing the sharpness of \Cref{thm:CI-SPR}} \label{ssec:eg_not_spr}
Here we give a class of examples where (\hyperref[UCS]{UCS}) holds, but $\delta_{\phi, \infty} + h_\infty=0$ and (\hyperref[def:SPR]{SPR}) fails, so that the condition  $\delta_{\phi, \infty} + h_\infty<0$ in  \Cref{thm:CI-SPR} is necessary.

Let $a(n)= 2^n$ and $C, \beta>0$ to be chosen later.  Now define  $\phi|_{[rr]}=\log{C}$, $\phi|_{[rv_1^{n, i}]} = \log{C}-n\log 2- \beta \log n$ and $\phi=0$ otherwise (as in \Cref{rmk:spread} we could also spread this potential out if desired).  

First observe that $Z_n(\phi, [r])\ge C2^n2^{-n}n^{-\beta}$, so $P(\phi)\ge 0$.

Taking the first return map to $[r]$ the induced potential $\overline{\phi}$ corresponding to loops of length $n$ takes the value $\log C-n\log 2-\beta\log n$.  Then 
$$P(\overline{\phi})=\log\left(C\sum_n a(n)e^{-n\log 2-\beta\log n}\right) 
=\log\left(C\zeta(\beta)\right)$$ where $\zeta$ denote the Riemann zeta function.  We use the ideas of Hofbauer and Keller presented in \cite[Section 4.1]{IomTod13}, generalised to this setting (see also the ideas of \cite[Table 1]{Rue03}). 
\begin{itemize}
\item[(a)]  If $\beta>1$ and we choose $C = 1/\zeta(\beta)$ then the pressure of the induced system is zero, $\phi$ is recurrent and  $P(\phi)=0$. 
\item[(b)] If $\beta>1$ and $C>  1/\zeta(\beta)$, or $\beta\in (0, 1]$, then the pressure of the induced system is positive and this is not interesting for our purposes (note this would imply $P(\phi)>0$).  
\item[(c)] If $\beta>1$ and  $C< 1/\zeta(\beta)$ then $\phi$ is transient and $P(\phi)=0$.
\end{itemize}
We will now assume that we are in case (a).

Since 
$$C\sum_n na(n)e^{-n\log 2-\beta\log n} = C\sum_n n^{1-\beta},$$
the system is positive recurrent, and we have an equilibrium state $\mu_\phi$ here, if $\beta>2$ (if $\beta\in (1, 2]$ then $\phi$ is null recurrent); moreover there is a conformal measure $m_\phi$.  Since $h_{top}(\sigma)$ must solve 
$$1=\sum_n a(n)e^{-nh_{top}(\sigma)} = \sum_n 2^ne^{-nh_{top}(\sigma)},$$ we see that $h_{top}(\sigma) = \log 4$.

The fact that $h_\infty=\log 2$ follows from \Cref{lem:hinfa(n)}.    We next show that  $\delta_{\phi, \infty} = -\log 2$.  That $z_{n, \phi}(M, q)\ge -\log 2+\frac1n(\log C+\beta\log n)$ for $n+1>M$ and $n>n_q$ is immediate from the definition, so $z_{n, \phi}(M, q)\ge -\log 2$.  For the upper bound, the proof is similar to, though simpler than, that of \Cref{lem:hinfa(n)}: if we consider $v\in \mathcal{G}(q)$ as defined there, then for $x\in [v]$, $\frac1{|v|}S_{|v|}\phi(x)\le  -\log 2$   and since the finite behaviour contributed by any prefixes and suffixes disappears in the limit,  $\delta_{\phi, \infty} = -\log 2$ and so $\delta_{\phi, \infty} + h_\infty= 0$.

We see here that $Z_n^*(\phi, r)= C/n^\beta$ so (\hyperref[def:SPR]{SPR}) fails.
Hence \Cref{thm:CI-SPR} is sharp in the sense that we can satisfy (\hyperref[UCS]{UCS}), but if  $\delta_{\phi, \infty} + h_\infty< 0$ does not hold then  we can fail (\hyperref[def:SPR]{SPR}).  Note also that if $\phi$ was null recurrent or, as in case (c) above, transient, we would also fail these conditions in a more dramatic way.

\subsection{Infinite entropy case}

If $h_{top}(\sigma)=\infty$ then $h_\infty= \infty$, so $h_\infty+\delta_{\phi, \infty}<0$ doesn't make sense, and anything can happen.  For example, take $a(n)=2^{2^n}$ and set $\phi|_{[rr]}=\log C$, $\phi|_{[rv_1^{n, i}]} = \log C-2^n\log 2- \psi(n)$ and $\phi=0$ otherwise, where $\limsup_n\frac1n\log|\psi(n)|<\infty$ and $C$ is chosen so that $P(\overline{\phi})=0$.  Then we can easily ensure (\hyperref[UCS]{UCS}) by making $\psi$ not large for small $n$, but we can also pick $\psi(n)$ so that the system satisfies SPR, or choose it so that it does not.  

\subsection{Relation of bouquets to inducing schemes and general shifts}
\label{ssec:bouqind}

At the beginning of this section we described interval maps $(I, f)$ with an inducing scheme $(Y=\cup_iY_i, F=f^\tau)$.  
If we have $F(Y_j)=Y$ for all $i$, which is the case for the examples mentioned above, then we identify $Y$ with $r$, let $a(n)$ be the number of $j$ with $\tau_j=n$, and associate a loop $r\mapsto v_1^{n, i_j} \mapsto  v_2^{n, i_j}\mapsto \cdots\mapsto v_{n-1}^{n, i_j}\mapsto r$ with each such $j$.

We can project a sequence $(x_0, x_1, \ldots)\in \Sigma$ to $x\in I$ by a projection $\pi$ as follows.  Suppose that $x\in Y$ has $F^\ell(x)\in [rv_1^{n_\ell, i_\ell}]$ for all $\ell\ge 0$ some $n_\ell, i_\ell$.  Then there will be a corresponding sequence $(x_0, x_1, \ldots)\in \Sigma$ given by $(r, v_1^{n_0, i_0}, v_2^{n_0, i_0}, \ldots, v_{n_0-1}^{n_0, i_0}, r, v_1^{n_1, i_1}, \ldots)$.  So let $\pi(x_0, x_1, \ldots) = x$ here.  If $x_0= v_k^{n, i_j}$ for $k>1$ then consider $y\in Y$ the projection of the sequence $(r,v_1^{n, i_j}, \ldots, v_{k-1}^{n, i_j}, x_0, x_1, \ldots)$ and let $\pi(x_0, x_1, \ldots) = f^k(y)$.

If $\phi:I\to [-\infty, \infty]$ is a potential, then this lifts to a potential on the bouquet shift $ \phi\circ\pi$.  The regularity of the lifted potential depends on the regularity of the original one and the choice of inducing scheme.  
For some specific cases of multimodal maps where $\phi=-\log|Df|$ and there is an inducing scheme so that $\phi$  lifts to a potential of summable variation, see for example \cite[Proposition 4.1]{BruLuzStr03} which considers multimodal maps with different rates of growth of derivative along critical orbits.  In this case Collet-Eckmann maps yield symbolic models satisfying (\hyperref[UCS]{UCS}) along with our other equivalent properties, while non-Collet-Eckmann maps fail all of these.

We can extend a version of the coding used above for any topologically transitive CMS $(\Sigma, \sigma)$:  we can pick a 1-cylinder and take first returns to it and then use the induced system to recode the system via a bouquet with the root being the 1-cylinder selected.  Hence the bouquet setup captures the behaviour of \emph{any} topologically transitive CMS. 

\subsection{Cases where the measure of maximal entropy satisfies SPR}

We close by noting that Theorem~\ref{thm:CI-SPR} implies that if it is known that $h_\infty<h_{top}$, which means that the measure of maximal entropy is SPR (see \cite[Proposition 6.1]{Buz10}, \cite[Proposition 2.20]{IomTodVel22}), then also the equilibrium state for a potential $\phi$ with $\sup\phi-\inf\phi< h_{top}-h_\infty$ must also be SPR since this automatically implies $\delta_{\phi, \infty} +h_\infty<P(\phi)$. 
This is shown to be the case for interval maps in \cite{BruTod08}. 
There are various cases of systems which have a coding by a countable Markov shift and where it may, in the future, be proved that the measure of maximal entropy is SPR the above idea would then apply.  For example we might expect the surface diffeomorphisms considered in \cite{BuzCroSar22} to satisfy these conditions.

 \end{document}